\documentclass[a4paper,12pt]{article}
\usepackage{latexsym}
\usepackage{amsthm,amssymb,amsfonts,mathrsfs,amsmath}%\mathbb,proof
\usepackage{a4wide,fullpage}
\usepackage[colorlinks,citecolor=blue,urlcolor=blue]{hyperref}
\usepackage{graphicx}
\usepackage{enumerate}
\usepackage{bm}
\usepackage{color}
\usepackage{natbib}
\usepackage{tikz}
\usepackage{subcaption}

\newcommand{\R}{\mathbb R}
\newcommand\Prob{{\mathbb P}}
\newcommand\E{{\mathbb E}}
\newcommand\V{{\rm Var}}

\newtheorem{theorem}{Theorem}

\newtheorem{lemma}{Lemma}

\newtheorem{example}{Example}
\newtheorem{proposition}{Proposition}
\theoremstyle{remark}

\newcommand\inprob{\buildrel {\mathrm{P}} \over  \longrightarrow}
\newcommand\convlaw{\buildrel {\mathrm{D}} \over \longrightarrow }

\allowdisplaybreaks

\title{Asymptotic Normality for Triangle Counting in the Sparse $\beta$-model}
\author{Siang Zhang, Qunqiang Feng\thanks{Corresponding author email: fengqq@ustc.edu.cn},  
and Zhishui Hu\\
{\small Department of Statistics and Finance, School of Management} \\
{\small University of Science and Technology of China}\\
{\small Hefei 230026, China}}

\date{}
\begin{document}
\maketitle

\begin{abstract}
	
We study the number of triangles $T_n$ in the sparse $\beta$-model
on $n$ vertices, a random graph model that captures degree heterogeneity in real-world networks.
Using the norms of the heterogeneity parameter vector, 
we first determine the asymptotic mean and variance of $T_n$. 
Next, by applying the Malliavin–Stein method, we derive a non-asymptotic upper bound on the Kolmogorov distance between normalized $T_n$
and the standard normal distribution. 
Under an additional assumption on degree heterogeneity,
we further prove the asymptotic normality for $T_n$, as $n\to\infty$. 

\vskip 0.2cm
\noindent{\it Keywords:} Random networks, sparsity, degree heterogeneity, subgraph counts, Malliavin-Stein method
\vskip 0.2cm
	
\noindent{\it AMS 2020 Subject Classification}: Primary
05C80,    % random graphs
90B15;    % network models, stochastic
Secondary
%60F10,    % Large deviations
60F05   % Central limit and other weak theorems
\end{abstract}

\section{Introduction}

Subgraph counting is a fundamental problem in the study of random graphs and has been extensively investigated over several decades. 
A significant body of research has focused on this topic within the framework of classical Erd\H{o}s-R\'{e}nyi (ER) random graphs 
(see, e.g., \cite{janson2000random,bollobas2001} and references therein). 
For instance, as the graph size (i.e., the number of vertices) tends to infinity, 
the necessary and sufficient conditions for the asymptotic normality of the count of any fixed subgraph in the ER random graph model
were first established in the seminal work of Ruci\'nski \cite{rucinski1988whenas}. 
More recently, this result has been complemented with explicit bounds on the Kolmogorov distance in \cite{Privault2020}; 
see also \cite{rollin2022,eichelsbacher2023Kol,eichelsbacher2023simplified}.

Beyond random graph theory, subgraph counting serves as a powerful tool in network science, 
enabling the quantification of specific subgraphs (e.g., triangles, cycles, cliques
and trees) to uncover the structural and functional properties of complex systems (see, e.g., \cite{milo2002}). 
Applications range from identifying subgraphs in biological networks \cite{alon2007} to analyzing the dynamics of functional brain networks \cite{bullmore2009}, providing insights into both local and global organization in real-world networks. 
However, subgraph counting is computationally intensive, particularly for large networks (see, e.g., \cite{Teixeira2015}). 
For a comprehensive overview of subgraph counting methods and efficient algorithms in network science, 
we refer to \cite{Ribeiro2021}.

Although the ER random graph model has played a pivotal role in the development of random graph theory and network science, 
its simplicity and restrictive assumptions limit its applicability to model real-world networks. 
Real-world networks exhibit complex topological features that are not captured by the ER random graph model \cite{newman2018}. 
One such feature is degree heterogeneity \cite{Estrada2010}, which refers to the variability in vertex degrees within a network. 
Unlike ER random graphs, where vertex degrees are concentrated around the average degree, 
many real-world networks exhibit heavy-tailed degree distributions, often approximated by power-law behavior (see, e.g., \cite{barabasi1999,clauset2009}).  
This indicates the presence of a few highly connected hubs and many sparsely connected vertices.
Another key feature is the clustering coefficient, 
which quantifies the tendency of nodes to form tightly knit groups \cite{newman2001}. 
In real-world networks, the clustering coefficient is typically high and remains relatively constant as the network grows, 
whereas it is notably low in ER graphs with comparable edge density (see, e.g., \cite{newman2009random}). 
In other words, real-world networks are usually sparse but contain a large number of triangles.
For example, in social networks, the probability that two friends of a given individual are also friends with each other is significantly higher than predicted by the ER random graph model \cite{watts1998}.

The $\beta$-model is a widely studied statistical network model that incorporates vertex-specific parameters to capture degree heterogeneity in real-world networks \cite{chatterjee2011randomgw}. 
It belongs to the broader class of exponential random graph models, which describe the probability of observing a given network structure (see, e.g., \cite{holland1981, robins2007}). For statistical inference on the parameters of the $\beta$-model and its variants, we refer to \cite{Rinaldo2013,Yan2013clt,Karwa2016,chen2021,Chang2024}.

To the best of our knowledge, in contrast to ER random graphs with the regularity of vertex degrees which could facilitate analysis, 
there are only a few theoretical results on subgraph counting in heterogeneous network models. 
Several limit laws of the clustering coefficient, which is closely related to the number of triangles, 
in the configuration model and  rank-1 inhomogeneous
random graph model for scale-free networks have been
established (see Theorems 1.3--1.5 in \cite{vanderhofstad2020assortativity}).
Employing the generalized $U$-statistics, 
the number of fixed-size cliques (where triangles are a special case) in graphon based random graphs
has also been considered in \cite{Hladky2021}; see \cite{bhattacharya2023fluctuations} 
for general subgraph counts. Owing to the inhomogeneous structure of these models,  
these limit theorems do not yield explicit rates of convergence.
In this paper, we aim to derive the asymptotic properties of triangle counts in the sparse $\beta$-model as the number of vertices tends to infinity. In particular,
using the Malliavin-Stein method \cite{nourdin_peccati_2012}, 
under several mild conditions we establish the asymptotic normality of triangle counts with
an explicit Berry-Esseen bound. 

Throughout this paper, we shall use the following notation. 
Let $[n]$ denote the set $\{1,2,\dots,n\}$, for any positive integer $n$.
We use $C > 0$ to denote a generic constant which may vary from one occasion to another.
For a finite set $S$, let $|S|$ be the cardinality of $S$.
For a vector $\bm{x}=(x_1,x_2,\dots,x_n)\in \mathbb{R}^n$, 
denote by $\|\bm{x}\|_s=(|x_1|^s+|x_2|^s+\dots+|x_n|^s)^{1/s}$ the $L_s$-norm of $\bm{x}$ for $s>0$,
and $x_{\max}$ and $x_{\min}$ the maximal and minimal entry of $\bm{x}$, respectively.  

The rest of this paper is organized as follows. 
In Section 2, we first formally introduce the $\beta$-model, 
and then state our main result demonstrating the asymptotic normality of 
triangle counts under mild sparsity and heterogeneity conditions.
In Section 3, we derive the asymptotic mean and variance of triangle count under the sparsity condition. 
Finally, with an adaption of the Malliavin-Stein method we prove our main results in
Section 4.

\section{Model description and main results}\label{sec:mainresults}

The $\beta$-model is formally defined as follows \cite{chatterjee2011randomgw}.
Consider a random graph with vertex set $[n]$,
where $n\ge 2$ is an integer.
The presence of an edge between vertices $i$ and $j$ is determined independently with probability
\begin{align} \label{pijbeta}
p_{ij}= \frac{e^{\beta_i + \beta_j}}{1+ e^{\beta_i + \beta_j}},  \quad 1\le i \neq j\le n,
\end{align} 
where $\beta_{i}\in \mathbb{R}$ is the degree heterogeneity parameter for vertex $i$. 
We note that $\beta_i$ can be negative and may depend on $n$.
These parameters quantify the connectivity propensity of vertices:
 a larger $\beta_{i}$ corresponds to a higher likelihood of vertex $i$ forming edges.
Notably, when $\beta_{i}=\beta$ is a constant for all $i\in [n]$, 
the $\beta$-model reduces to the ER random graph model with edge probability $p={\rm e}^{2\beta}/(1+{\rm e}^{2\beta})$. 

Let $I_{ij}$ denote the indicator for the presence
of an edge between vertices $i$ and $j$. By construction, 
$I_{ii}=0,I_{ij}=I_{ji}$, 
and the collection $\{I_{ij},1\le i<j\le n\}$ consists of independent Bernoulli random variables with success rates $p_{ij}$.
Clearly, the adjacency matrix ${\bm A}=(I_{ij})_{n\times n}$ 
is symmetric with zero diagonal.
In terms of ${\bm A}$, the number of triangles $T_n$ in the $\beta$-model on $n$ vertices is given by
\begin{align}\label{eq:Tn1}
T_n = \frac16 \text{tr}({\bm A}^3)= \sum_{1\le i<j<k\le n} I_{ij} I_{jk} I_{ki},
\end{align} 
where $\text{tr}(\cdot)$ denotes the matrix trace.

For simplicity of notation, we define $\mu_i={\rm e}^{\beta_i}>0$ for each $i\in [n]$ and let
\[\bm{\mu}=(\mu_1,\mu_2,\dots,\mu_n)^\top,\]
omitting the subscript $n$.
The edge probabilities in \eqref{pijbeta} then become
\begin{align}\label{pijmu}
p_{ij}= \frac{\mu_i\mu_j}{1+ \mu_i\mu_j}. 
\end{align} 
%Let $\mu_{\max}=\max\{\mu_1,\mu_2,\dots,\mu_n\}$ and $\mu_{\min}=\min\{\mu_1,\mu_2,\dots,\mu_n\}$.
Deriving asymptotic properties of $T_n$ under general parameters $\beta_i$ presents significant challenges due to degree heterogeneity.
To address this, we impose the following sparsity conditions:   
\begin{align}  \label{eq:condion1}
\mu_{\max}\to 0 \quad \mbox{and} \quad  \|\bm\mu\|_2 \rightarrow \infty.
\end{align}
The first condition ensures network sparsity, 
while the second guarantees a non-degenerate graph structure with a substantial number of triangles 
(see Proposition \ref{prop:mv} below).
Note that for any $s>t>0$,
\[
\|\bm\mu\|_s^s\le \mu_{\max}^{\,s-t}\|\bm\mu\|_t^t,
\]
since each $\mu_i$ is positive. It thus follows by the assumption $\mu_{\max}\to0$ in \eqref{eq:condion1} that
\begin{equation}\label{normorder}
    \|\bm\mu\|_s^s = o\big(\|\bm\mu\|_t^t\big), \quad  s>t>0.
\end{equation}
 
The Kolmogorov distance between random variables $X$ and $Y$ is defined as
\[
d_K(X,Y)=\sup_{z\in\R}|\Prob(X\leq z)-\Prob(Y\leq z)|.
\]
Define the normalized triangle count $F_n$ as
\begin{equation}\label{DefFn}
F_n= \frac{T_n - \mathbb{E}[T_n]}{\sqrt{\V[T_n]}}.    
\end{equation}
The following theorem provides an upper bound for $d_{K}(F_n, \mathcal{N})$,
where $\mathcal{N}$ denotes a standard normal random variable.

\begin{theorem}\label{Theorem:main} 
Under the sparsity condition \eqref{eq:condion1}, we have
\begin{align*}
  d_{K}(F_n, \mathcal{N}) \le\frac{C}{\|\bm\mu\|_2^{\frac{5} {2}}\big(\|\bm\mu\|_3^6+\|\bm\mu\|_2^2\big)}\sum_{\ell=1}^5A_{\ell}, 
\end{align*}
where $C>0$ is a constant, and
\begin{equation}\label{def:Aell}
\begin{aligned}
A_1&=\|\bm\mu\|_{\frac{3}{2}} ^ {\frac{3}{4}} \|\bm\mu\|_{\frac{5}{2}}^{\frac{5}{4}}, \quad
A_2=\|\bm\mu\|_{\frac{3}{2}}^{\frac{3}{4}}\|\bm\mu\|_{2}^{\frac12}\|\bm\mu\|_{\frac{7}{2}}^{\frac{7}{4}}\|\bm\mu\|_{4}^{2},\\
A_3&=\|\bm\mu\|_2^{\frac{5}{2}}\|\bm\mu\|_5^{5},\quad
A_4=\|\bm\mu\|_{2}^{\frac{3}{2}} \|\bm\mu\|^{\frac{5}{2}}_{\frac{5}{2}}\|\bm\mu\|_{5}^{\frac{5}{2}},\quad 
A_5=\|\bm\mu\|_{\frac{7}{4}}^{\frac{7}{4}}\|\bm\mu\|_{\frac{7}{2}}^{\frac{7}{4}}.
\end{aligned}   
\end{equation}
\end{theorem}

The bound in  Theorem \ref{Theorem:main}  is complicated due to the generality of  $\bm\mu$.  
The following example demonstrates that in specific cases, 
this bound matches the performance of existing results.

\begin{example} 
    Fix a positive integer $K$ and a sequence of positive constants $\{\pi_r, 1\le r\le K\}$
satisfying  $\sum_{r=1}^K\pi_r=1$.
Let $\{V_r, 1\le r\le K\}$ be a partition of $[n]$ with the cardinality $|V_r|=\pi_r n+o(n)$. 
Assume that for each $i\in V_r$ with $1\le r\le K$,
\[\mu_i=\theta_r\,n^{-\alpha/2},\] 
where the constants $\theta_r>0$ and $0<\alpha<1$.
Then we have that $\mu_{\max}$ is of order $n^{-\alpha/2}$, and for any fixed $s>0$,
\[
\|\bm\mu\|_s^s=\sum_{r=1}^K |V_r|\,\theta_r^s n^{-\alpha s/2}
= n^{1-\alpha s/2}\Big(\sum_{r=1}^K \pi_r\theta_r^s\Big)(1+o(1)).
\]
In particular, when $s=2$, the quantity $\|\bm\mu\|_2^2$ is of order $n^{1-\alpha}$.

This blockwise-constant scaling is standard in sparse stochastic block models and their degree-corrected variants; see, e.g., \cite{ holland1981,KarrerNewman2011,mosselneemansly2015,abbe2018} for more background.
An application of Theorem \ref{Theorem:main} together with some basic calculations yields that
\begin{equation*}
   d_K(F_n,\mathcal N)\le C\, n^{-\eta(\alpha)},  
\end{equation*}
where
\[
\eta(\alpha)=
\begin{cases}
1-\alpha, & 0<\alpha\le \tfrac12,\\[0.2em]
\tfrac34-\tfrac{\alpha}2, & \tfrac12< \alpha\le \tfrac23,\\[0.2em]
\tfrac54(1-\alpha), & \tfrac23<\alpha<1.
\end{cases}
\]

In the special case $K=1$, we have that $\mu_i\equiv\mu=cn^{-\alpha/2}$ for all $i\in [n]$, 
where $c>0$ and $\alpha\in (0,1)$. Then our model reduces to ER
random graph model with edge probability
\begin{equation*}
  p= \frac{\mu^2}{1+ \mu^2}\sim c^2 n^{-\alpha}
\end{equation*}
as $n\rightarrow\infty$. 
In this case, the order $\eta(\alpha)$ coincides with that established in \cite[Theorem 1.1]{kai2017discretemalliavin}. 
\end{example}

Notably, the upper bound in Theorem \ref{Theorem:main} may not vanish as $n\to\infty$
under the sparsity condition \eqref{eq:condion1}, 
primarily due to the $L_{3/2}$-norm of $\bm\mu$ term in $A_1$ and $A_2$, 
which can grow very fast.
To establish asymptotic normality for $T_n$, 
we introduce two additional heterogeneity conditions:
\begin{equation}\label{heterocond}
\|\bm\mu\|_{\frac{3}{2}} ^ {\frac{3}{2}}=O\big(\|\bm{\mu}\|_2^6\big),
\end{equation}
and
\begin{align}\label{eq:condition2}
\frac{\mu_{\max}}{\mu_{\min}} = O\big(\|\bm{\mu}\|_2^\frac{3}{2}\big).
\end{align}
Condition \eqref{heterocond} directly restricts the growth rate of the $L_{3/2}$-norm of $\bm\mu$.  
Condition \eqref{eq:condition2} accommodates considerable degree heterogeneity in the $\beta$-model, 
thereby permitting the ratio $\mu_{\max}/\mu_{\min}$ to diverge at an appropriate rate. 
This is in sharp contrast to the ER random graph model, where this ratio is invariably 1.
Such regularity conditions are widely adopted in statistical network analysis. 
For instance, in community detection problems, 
several constraints on the extreme values of degree heterogeneity parameters are imposed to bound the estimation error (see, e.g. \cite{jin2015}).

The following result establishes the asymptotic normality of $T_n$ under these conditions.

\begin{theorem}\label{Thm:AN}
Under the sparsity condition \eqref{eq:condion1}, if either \eqref{heterocond} or \eqref{eq:condition2} holds,
then as $n\to\infty$,
\begin{align*}
F_n \convlaw \mathcal{N},
\end{align*}    
where $\convlaw$ denotes the convergence in distribution.
\end{theorem}

\section{Mean and variance}

In this section, we derive the asymptotic mean and variance of $T_n$, 
the number of triangles in the \(\beta\)-model, 
under the sparsity condition \eqref{eq:condion1}. 

\begin{proposition}\label{prop:mv}
Suppose that the sparsity condition \eqref{eq:condion1} holds. 
As $n\to\infty$, we have
\begin{align}
\E[{T_n}]&= \frac16\|\bm\mu\|_2^6(1+o(1)),     \label{ETn} \\
\V[{T_n}]&= \frac16\|\bm\mu\|_2^4\big(3\|\bm\mu\|_3^6+\|\bm\mu\|_2^2\big)(1+o(1)). \label{VarTn}
\end{align}
\end{proposition}

\begin{proof}
We first note that, by \eqref{pijmu} and \eqref{eq:condion1}, 
the edge probability $p_{ij}$ satisfies 
\begin{align}\label{eq:pij}
p_{ij} = (1 + o(1)) \mu_i \mu_j, \quad 1\le i\neq j\le n.
\end{align}

For distinct vertices $i,j,k\in[n]$, by definition of the \(\beta\)-model,
the indicators $I_{ij},I_{jk}$ and $I_{ki}$ are independent Bernoulli random variables.  
Using \eqref{eq:Tn1} and \eqref{eq:pij}, the expectation of $T_n$ is equal to
\begin{align}\label{eq:ETn}
\E[{T_n}]= \sum_{1\le i<j<k\le n} p_{ij} p_{jk} p_{ki}=(1+o(1))\sum_{1\le i<j<k\le n} \mu_i^2\mu_j^2\mu_k^2.
\end{align} 
In term of norms of $\bm\mu$, the summation on the right-hand side of \eqref{eq:ETn} can be rewritten as 
\begin{align*}
\sum_{1\le i<j<k\le n} \mu_i^2\mu_j^2\mu_k^2
&=\frac16\sum_{1\le i,j,k\le n} \mu_i^2\mu_j^2\mu_k^2- \frac12\sum_{1\le i,j\le n} \mu_i^2\mu_j^4+\frac13\sum_{i=1}^{n} \mu_i^{6} \notag\\
&=\frac16\|\bm\mu\|_2^6-\frac12\|\bm\mu\|_2^2\|\bm\mu\|_4^4+\frac13\|\bm\mu\|_6^6.
\end{align*}
Thus, combining this with \eqref{normorder} and \eqref{eq:ETn} yields \eqref{ETn}.

We now turn to the variance of $T_n$. To this end, we define
\[
X_{ij} = I_{ij} - p_{ij}, \quad 1\le i\neq j\le n.
\]
Then, for any $1\le i\neq j\le n$, we have that $\E[X_{ij} ]=0$, and by \eqref{eq:pij},
\begin{align}\label{VXij}
\V[X_{ij}]=\E[X_{ij}^2]= p_{ij}(1- p_{ij})=(1 + o(1)) \mu_i \mu_j.
\end{align}
By \eqref{eq:ETn}, we can now reformulate \eqref{eq:Tn1} as
\begin{align}
T_n &= \sum_{1\le i<j<k\le n} (X_{ij} + p_{ij})(X_{jk} + p_{jk})(X_{ki} + p_{ki})  \notag \\
    &= \E[{T_n}]+\sum_{1\le i<j\le n}c_{ij}X_{ij}
       +\sum_{\substack{1\le i<j\le n\\ k\neq i,j}}p_{ij}X_{jk} X_{ki}
       +\sum_{1\le i<j<k\le n}X_{ij} X_{jk} X_{ki}, \label{Eq:TnX}
\end{align}
where, by \eqref{eq:pij},
\[
c_{ij}=\sum_{k\ne i,j}p_{jk} p_{ki}=(1+o(1))\mu_i\mu_j\sum_{k\ne i,j}\mu_k^2=(1+o(1))\|\bm\mu\|_2^2\mu_i\mu_j.
\]
The $1+\binom{n}{2}+(n-2)\binom{n}{2}+\binom{n}{3}$ terms on the right-hand side of \eqref{Eq:TnX} are pairwise uncorrelated. 
By \eqref{eq:pij} and \eqref{VXij}, we have
\begin{align*}%\label{VarTn1}
    \V[T_n]&= \sum_{1\le i<j\le n}c_{ij}^2\E[X_{ij}^2]
                    +\sum_{\substack{1\le i<j\le n\\ k\neq i,j}}p_{ij}^2\E[X_{jk}^2]\E[X_{ki}^2] 
                    +\sum_{1\le i<j<k\le n}\E[X_{ij}^2]\E[X_{jk}^2]\E[X_{ki}^2]\\
    &=(1+o(1)) \left(\|\bm\mu\|_2^4\sum_{1\le i<j\le n} \mu_i^3\mu_j^3+\|\bm\mu\|_2^2\sum_{1\le i<j\le n} \mu_i^3\mu_j^3+\sum_{1\le i<j<k\le n} \mu_i^2\mu_j^2\mu_k^2\right).
\end{align*}
Since $\|\bm\mu\|_2^2=o(\|\bm\mu\|_2^4)$ and
\begin{align*}
\sum_{1\le i<j\le n} \mu_i^3\mu_j^3=\frac12\big(\|\bm\mu\|_3^6-\|\bm\mu\|_6^6\big),
\end{align*}
by \eqref{ETn} and \eqref{eq:ETn} we can proceed  with  
\begin{align*}
  \V[T_n]&=\Big(\frac12\|\bm\mu\|_2^4\big(\|\bm\mu\|_3^6-\|\bm\mu\|_6^6\big)
  +\frac{1}{6}\|\bm\mu\|_2^6\Big)(1+o(1))  \\
                &=\frac16\|\bm\mu\|_2^4\big(3\|\bm\mu\|_3^6-3\|\bm\mu\|_6^6+\|\bm\mu\|_2^2\big)(1+o(1)), 
\end{align*}
which, together with \eqref{normorder}, completes the proof of \eqref{VarTn}.
\end{proof}

An immediate consequence of Proposition \ref{prop:mv} is in the following.

\begin{proposition}
Under the sparsity condition \eqref{eq:condion1},  as $n\to\infty$,
\[
\frac{T_n}{\|\bm\mu\|_2^6}\inprob\frac16,
\]
where $\inprob$ denotes the convergence in probability.
\end{proposition}

\begin{proof}
Applying Proposition \ref{prop:mv} and Chebyshev's inequality, for any $\varepsilon>0$, by \eqref{normorder} we have
\[
\Prob\Big(\Big|\frac{T_n}{\E[T_n]}-1\Big|>\varepsilon\Big)\le \frac{\V[T_n]}{(\varepsilon\E[T_n])^2}
=\frac{6\big(3\|\bm\mu\|_3^6+\|\bm\mu\|_2^2\big)}{\varepsilon^2\|\bm\mu\|_2^8}
(1+o(1))\to 0,
\]
which implies that $T_n/\E[T_n]$ converges to 1 in probability.
Thus, the desired result follows via Slutsky's theorem and \eqref{ETn}. 
\end{proof}

\section{Proofs}

This section is devoted to the proof of our main results stated in Section~\ref{sec:mainresults}. 
We first introduce several auxiliary lemmas essential to our approach.

Our primary tool for establishing the asymptotic normality of triangle counts in the $\beta$-model is the Malliavin-Stein method \cite{nourdin_peccati_2012},
a powerful synthesis of Malliavin calculus \cite{nualart2006} and Stein's method (see, e.g., \cite{barbour2005introduction}). 
This method is particularly effective for analyzing normal approximations when applied to independent
Rademacher or Bernoulli random variables \cite{gesine2010steinrademacher,krokowski2017rademacher}. 
Especially, applying this method, Berry-Esseen bounds for triangle counts in ER random graphs have been established \cite{kai2017discretemalliavin}, 
while non-uniform Berry-Esseen bounds for counts of any fixed subgraph have been further derived in \cite{butzek2024nonuni}. 
For more background and related bounds in the Malliavin--Stein framework, 
see, e.g., \cite{eichelsbacher2023simplified} and references therein.

To state the normal-approximation bound used in our approach (Lemma \ref{lem:Eich} below), 
we introduce, following \cite{eichelsbacher2023simplified}, 
the discrete gradient operator for functionals of independent Bernoulli random variables. 
For our purposes, only the first- and second-order discrete gradients are needed.

Let $m \ge 3$ be an integer, and $\bm{X}=(X_1,\ldots,X_m)$ random vector of independent Bernoulli random variables with
\[
\Prob(X_a = 1)=p_a \quad\text{and}\quad  \Prob(X_a = 0)=q_a,\quad a\in[m],
\]
where $0<p_a<1$ and $q_a=1-p_a$. For a measurable function $f:\{0,1\}^m\to\mathbb{R}$ and $F=f(\bm X)$, define the discrete gradient of $F$ with respect to $X_a$ by
\begin{align} \label{eq:DaF1}
D_aF=\sqrt{p_aq_a}\,\big(f(\bm{X}_a^+)-f(\bm{X}_a^-)\big),
\end{align}
where
\[
\bm{X}_a^+=(X_1,\ldots,X_{a-1},1,X_{a+1},\ldots,X_m),\qquad
\bm{X}_a^-=(X_1,\ldots,X_{a-1},0,X_{a+1},\ldots,X_m).
\]
Using \eqref{eq:DaF1}, one can further define the second-order discrete gradient 
$D_bD_aF:=D_b(D_aF)$. Note that for any $a,b\in [m]$,
\[
D_bD_aF=D_aD_bF.
\]

The following lemma plays an important role in our proof of Theorem \ref{Theorem:main}.

\begin{lemma}
\label{lem:Eich} 
 Let $F=f(\bm{X})$ be a random variable with $\E[F]=0$ and  $\V[F]=1$. Then,
\begin{align*}
	d_{K}(F, \mathcal{N})  \leq C\sum_{k=1}^5 \sqrt{B_k},
\end{align*}
where $\mathcal{N}$ is a standard normal random variable, and
\begin{align*}
B_1&=\sum_{a,b,c\in [m]}\sqrt{\E\big[D_a^2 D_b^2\big]\E\big[D_{ca}^2 D_{cb}^2 \big] }, \\
B_2&=\sum_{a,b,c\in [m]}\frac{1}{p_c q_c}\E\big[D_{ca}^2 D_{cb}^2 \big],  \qquad
B_3  = \sum_{a=1}^m \frac{1}{p_aq_a} \E\big[ D_a^4 \big],  \\
B_4&=\sum_{a,b\in [m]}\frac{1}{p_aq_a}\sqrt{\E\big[D_a^4\big]\E\big[D_{ab}^4\big]}, \qquad  
B_5 =\sum_{a,b\in [m]} \frac{1}{p_aq_a p_b q_b}  \E\big[  D_{ab}^4 \big],
\end{align*}
with $D_a:=D_aF$ and $D_{ab}:=D_aD_bF$.
\end{lemma}

\begin{proof}
This lemma can follow directly by applying Theorem 4.1(i) of \cite{eichelsbacher2023simplified} to the independent Rademacher random variables $\{Y_a=2X_a-1: a\in [m]\}$.
We only need to verify  condition~(3.3) in \cite{eichelsbacher2023simplified} and that
$u: a \mapsto (p_a q_a)^{-1/2} D_a F \, \big| D_a L^{-1} F \big|$ belongs to $\operatorname{Dom}(\delta)$,
where $\operatorname{Dom}(\delta)$ and the operator $L$ are defined 
as in Section~2 of~\cite{eichelsbacher2023simplified}.

 Note that $u$ depends only on 
finitely many independent Rademacher random variables $\{Y_a, a\in [m]\}$. 
Condition (3.3) in \cite{eichelsbacher2023simplified} is satisfied by Remark 3.2 (ii) of the same reference. Moreover, by the Wiener–It\^o–Walsh decomposition, the chaos expansion of $u$ contains only finitely many terms (see, for instance, \cite{Privault2008StochasticAO}, p. 444). Consequently, condition (2.14) in \cite{kai2017discretemalliavin} is satisfied, which implies  $u \in \operatorname{Dom}(\delta)$.
\end{proof}

To apply Lemma \ref{lem:Eich}, 
we denote the set of all possible edges 
in the $\beta$-model on $n$ vertices by $\{e_1,e_2,\dots,e_m\}$, where $m=\binom{n}{2}$. 
For each edge $e_a$ with $a\in[m]$, let $1\le i_a<j_a\le n$ denote its endpoints. 
Note that the edge indicators $\{I_{i_a j_a}, a\in [m]\}$ 
are independent Bernoulli random variables with 
\[
\Prob(I_{i_a j_a}=1)=p_{i_a j_a},\quad \Prob(I_{i_a j_a}=0)=1-p_{i_a j_a}.
\] 
Hence, the normalized triangle count $F_n$ in \eqref{DefFn} is a measurable function of these variables
and is thus amenable to analysis via Lemma~\ref{lem:Eich}.
 
By \eqref{eq:DaF1}, we have
\begin{align}\label{DaFn}
   D_a:= D_aF_n=\frac{\sqrt{h_a}}{\sqrt{\V[T_n]}}V_a,\quad a\in[m],
\end{align}
where 
\begin{equation}\label{def:h}
h_a=p_{i_aj_a}(1-p_{i_aj_a}),    
\end{equation}
and 
\begin{align}
\label{eq:Va}
  V_a = \sum_{k\neq i_a,j_a} I_{i_ak}I_{j_ak}.
\end{align}
In other words, $V_a$ counts the number of wedges (i.e., paths of length two) 
with endpoints $i_a$ and $j_a$. 
Graphically, the gradient $D_aF_n$ measures the sensitivity of triangle counts to flipping the status of $e_a$.

For the second-order gradient, consider any two given edges $e_a=\{i_a,j_a\}$ and $e_b=\{i_b,j_b\}$.
By definition, one can immediately obtain that $D_{ab}=0$ if either $e_a$ and $e_b$ are an identical edge or 
disjoint (i.e., the intersection $\{i_a,j_a\} \cap \{i_b,j_b\}$ is the empty set $\varnothing$).
If $|\{i_a,j_a\}\cap\{i_b,j_b\}|=1$, then \eqref{eq:DaF1} implies 
\begin{equation}\label{eq:DbVa}
D_bV_a=\sqrt{h_b}\,I_{\ell_1\ell_2},
\end{equation}
where
\begin{equation*}
\{\ell_1,\ell_2\}=(\{i_a,j_a\}\cup\{i_b,j_b\})\setminus(\{i_a,j_a\}\cap\{i_b,j_b\}),
\end{equation*}
i.e., $\ell_1$ and $\ell_2$ are the endpoints of $e_a$ and $e_b$ excluding their common vertex.
Consequently, it follows by \eqref{DaFn} and \eqref{eq:DbVa} that, for any $a,b\in [m]$,
\begin{align} \label{eq:Dab1}
D_{ab}:=D_aD_bF_n=D_bD_aF_n=\frac{\sqrt{h_ah_b}}{\sqrt{\V[T_n]}}\Delta_{ab},
\end{align}
with
 \begin{align}\label{defDelataab}
\Delta_{ab}=\left\{\begin{array}{cl} 
0,& \mbox{if~} \{i_a,j_a\}=\{i_b,j_b\} \mbox{~or~} \{i_a,j_a\}\cap\{i_b,j_b\}=\varnothing;\\
I_{\ell_1\ell_2}, & \mbox{if~} |\{i_a,j_a\}\cap\{i_b,j_b\}|=1. \end{array}\right.
\end{align}

In addition to Lemma \ref{lem:Eich}, we further require the following auxiliary lemmas.

\begin{lemma} \label{EDbT}
For any $s\ge2$, we have
\begin{align*}
\E[V_a^{s}] \le C \big( \|\bm\mu\|_2^2\mu_{i_a} \mu_{j_a} + \|\bm\mu\|_2^{2s}\mu_{i_a}^s \mu_{j_a}^s  \big),  
\end{align*}	
where $C=C(s)>0$ is a constant depending only on $s$. 
\end{lemma}

\begin{proof}
We begin with the inequality
\begin{equation*}
\E[|X+Y|^s]\le 2^{s-1}\E[|X|^s+|Y|^s], 
\end{equation*}
which holds for any random variables $X$ and $Y$. 
Applying this to $V_a$ given in \eqref{eq:Va}, we have
\begin{align}\label{EVa4ine}
\E[V_a^{s}] &  \le C\E\Big[\Big|\sum_{k\neq i_a,j_a} (I_{i_ak}I_{j_ak}-p_{i_ak}p_{j_ak})\Big|^s+\Big(\sum_{k\neq i_a,j_a} p_{i_ak}p_{j_ak}\Big)^s\Big].
\end{align}
To bound the first term on the right-hand side of \eqref{EVa4ine},
observe that $I_{i_ak}I_{j_ak}$ are independent Bernoulli variables with the success rate
$p_{i_ak}p_{j_ak}$ for $k\neq i_a,j_a$. 
By Rosenthal's inequality we have 
\begin{align*}
\E\Big|\sum_{k\neq i_a,j_a} (I_{i_ak}I_{j_ak}-p_{i_ak}p_{j_ak})\Big|^s
&\le C\Big[\sum_{k\neq i_a,j_a}\E|I_{i_ak}I_{j_ak}-p_{i_ak}p_{j_ak}|^{s}\\
&\quad+\Big(\sum_{k\neq i_a,j_a}\V[I_{i_ak}I_{j_ak}]\Big)^{s/2}\Big]\\
&\le C\Big[\sum_{k\neq i_a,j_a}p_{i_ak}p_{j_ak} +
\Big(\sum_{k\neq i_a,j_a} p_{i_ak}p_{j_ak}\Big)^{s/2}\Big].
\end{align*}
%Since 
%\[
%\E|I_{ik}I_{jk}-p_{ik}p_{jk}|^{s}=(1-p_{ik}p_{jk})^sp_{ik}p_{jk}+(1-p_{ik}p_{jk})p_{ik}^sp_{jk}^s
%\le 2p_{ik}p_{jk},
%\]
%and $\V[I_{ik}I_{jk}]\le p_{ik}p_{jk}$ for distinct $i,j,k$,
%it follows that
%\begin{align*}
%\E\Big|\sum_{k\neq i_a,j_a} (I_{i_a,k}I_{j_a,k}-p_{i_ak}p_{j_ak})\Big|^s
%\le  C\Big[\sum_{k\neq i_a,j_a}p_{i_ak}p_{j_ak} +
%\Big(\sum_{k\neq i_a,j_a} p_{i_a,k}p_{j_a,k}\Big)^{s/2}\Big].
%\end{align*}
Substituting this into \eqref{EVa4ine} and using the inequality 
$x^{s/2}\le x+x^s$ for $x>0$ and $s\ge2$,
by \eqref{eq:pij} we thus have 
\begin{align*}
\E[V_a^{s}] 
   &\le C\Big[\sum_{k\neq i_a,j_a} p_{i_ak}p_{j_ak}+\Big(\sum_{k\neq i_a,j_a} p_{i_ak}p_{j_ak}\Big)^{s}\Big] \notag \\ 
   &\le C\Big[\mu_{i_a}\mu_{j_a}\sum_{k\neq i_a,j_a} \mu_k^2+\Big(\mu_{i_a}\mu_{j_a}\sum_{k\neq i_a,j_a} \mu_k^2\Big)^{s}\Big] \notag \\
	& \le C\big(\|\bm\mu\|_2^2\mu_{i_a} \mu_{j_a} +\|\bm\mu\|_2^{2s}\mu_{i_a}^s \mu_{j_a}^s \big),
\end{align*}
which completes the proof of Lemma \ref{EDbT}. 
\end{proof}

The following inequality is used repeatedly in our analysis. 
\begin{lemma}\label{normsCS}
Let \(\bm{x} = (x_1, x_2, \dots, x_n)\) be a positive vector 
(i.e.,  $x_i>0$ for all $i\in [n]$). For any $s,t>0$, we have
\[
\|\bm {x}\|_{\frac{s+t}{2}}^{s+t} \le   \|\bm{x}\|_s^s \|\bm{x}\|_t^t.
\]
\end{lemma}

\begin{proof}
The inequality follows directly from the Cauchy-Schwarz inequality.    
\end{proof}

With the above preparation, we proceed to prove Theorem \ref{Theorem:main} below.

\begin{proof}[Proof of Theorem \ref{Theorem:main}]
Recall the second-order discrete gradients $D_{ab}$ given in \eqref{eq:Dab1} for all $a,b\in[m]$,
where $m=\binom{n}{2}$ represents the number of all possible edges in the $\beta$-model on $n$ vertices.
Observe that $\Delta_{ab}^s=\Delta_{ab}=\Delta_{ba}$ for all $s>0$, and $\Delta_{ca}$ and $\Delta_{cb}$
are independent for distinct edges $e_a,e_b,e_c$. 
Combining Lemma  \ref{lem:Eich} with \eqref{DaFn} and \eqref{eq:Dab1}, we obtain that
 \begin{align}\label{DKFn}
	d_{K}(F_n, \mathcal{N})  \leq \frac{C}{\V[T_n]}\sum_{k=1}^5 \sqrt{\widetilde B_k},
\end{align}
where 
\begin{align*}
\widetilde B_1:&=  \sum_{a,b,c\in [m]} h_ah_bh_c\sqrt{  \E\big[V_a^2 V_b^2 \big]  \E[\Delta_{ca}\Delta_{cb}] }, \\
\widetilde B_2:&=\sum_{a,b,c\in [m]} h_ah_bh_c\E[\Delta_{ca}\Delta_{cb}],\qquad
\widetilde B_3:= \sum_{a=1}^m h_a\E\big[ V_a^4 \big],\\
\widetilde B_4:&=\sum_{a,b\in [m]} h_ah_b\sqrt{\E\big[ V_a^4\big]\E[\Delta_{ab}]},\qquad 
\widetilde B_5:= \sum_{a,b\in [m]} h_ah_b \E[\Delta_{ab}].
\end{align*}
To prove Theorem \ref{Theorem:main}, by \eqref{VarTn} and \eqref{DKFn}
it suffices to show that
\[
\sum_{k=1}^5 \sqrt{\widetilde B_k}\le C\|\bm\mu\|_{2}^{\frac32}\sum_{\ell=1}^5A_{\ell},
\]
where $A_{\ell}>0\,(\ell=1,2,\dots, 5)$ are defined in \eqref{def:Aell}.
Applying the Cauchy–Schwarz inequality, a sufficient condition for this to hold is
\begin{align} \label{B_K}
\sum_{k=1}^5\widetilde B_k&\le C\|\bm\mu\|_{2}^3\sum_{\ell=1}^5A_{\ell}^2.
\end{align}
Thus, the remainder of the proof is devoted to verifying \eqref{B_K}.

We now analyze the terms \(\widetilde{B}_k\) individually, beginning with \(\widetilde B_3\).
From \eqref{eq:pij} and \eqref{def:h}, we deduce that
\begin{align*}
h_a=(1 + o(1)) \mu_{i_a} \mu_{j_a}, \quad a\in [m].
\end{align*}
Applying Lemma \ref{EDbT}  with \(s = 4\) yields %we thus have
\[
\widetilde B_3\le C \sum_{a=1}^m \big(\|\bm\mu\|_2^2\mu_{i_a}^2\mu_{j_a}^2 +\|\bm\mu\|_2^{8}\mu_{i_a}^5 \mu_{j_a}^5\big).
\]
Since for any $t>0$, 
\[
\sum_{a=1}^m \mu_{i_a}^t\mu_{j_a}^t
=\sum_{1\le i<j\le n}\mu_i^t\mu_j^t
=\frac12\Big[\Big(\sum_{i=1}^n\mu_i^t\Big)^2-\sum_{i=1}^n\mu_i^{2t}\Big]
\le \frac12\|\bm\mu\|_t^{2t},
\]
we have
\begin{align}\label{B3}
\widetilde B_3\le C  \big(\|\bm\mu\|_2^6+\|\bm\mu\|_2^8\|\bm\mu\|_5^{10}\big)
              =C\|\bm\mu\|_2^3(\|\bm\mu\|_2^3+A_3^2).
\end{align}

Next, consider \(\widetilde{B}_5\), 
which involves pairs of edges \(e_a=(i_a,j_a)\) and \(e_b=(i_a,j_b)\) sharing a common vertex \(i_a\).
%Summing over all such configurations is equivalent to summing over triples of distinct vertices \((i,j,k)\).
By \eqref{eq:pij} and \eqref{defDelataab}, we have
\begin{align} \label{B5}
\widetilde{B}_5 &=\sum_{i_a=1}^n\sum_{j_a\neq i_a}
\sum_{j_b\neq i_a,j_a}p_{i_aj_b}(1-p_{i_aj_a})p_{i_aj_b}(1-p_{i_aj_b})p_{j_aj_b}\notag\\
&=\sum_{i=1}^n\sum_{j\neq i}\sum_{k\neq i,j}p_{ij}(1-p_{ij})p_{ik}(1-p_{ik})p_{jk}\notag\\
&\le \sum_{i=1}^n\sum_{j\neq i}\sum_{k\neq i,j}p_{ij}p_{ik}p_{jk}\notag\\
&\le  \sum_{i,j,k\in[n]}\mu_i^2\mu_j^2\mu_k^2 = \|\bm\mu\|_2^6,
\end{align}
where in the second equality we relabeled the indices $i_a,j_a, j_b$ as $i,j, k$ for simplicity,
and in the last inequality we applied the bound $p_{ij}\le \mu_{i}\mu_j$.

Invoking  Lemma \ref{normsCS} with \(s = \frac{5}{2}\) and \(t = \frac{3}{2}\) gives
\[
\|\bm\mu\|_2^4 \le \|\bm\mu\|_{\frac{3}{2}} ^ {\frac{3}{2}} \|\bm\mu\|_{\frac{5}{2}}^{\frac{5}{2}}=A_1^2.
\]
Combining this with \eqref{B3} and \eqref{B5}, 
and noting that $\|\bm\mu\|_2^3=o(\|\bm\mu\|_2^4)$ under our assumptions, 
we have
\begin{align}\label{tildeB35}
\widetilde B_3+\widetilde B_5\le C \|\bm\mu\|_{2}^3(A_1^2+A_3^2).
\end{align}

For $\widetilde{B}_4$, Lemma \ref{EDbT} with $s=4$ provides
\begin{align}\label{sqrtEva4}
\sqrt{\E[V_a^{4}]} \le C \left( \|\bm\mu\|_2 \mu_{i_a}^{\frac{1}{2}}  \mu_{j_a}^{\frac{1}{2}} + \|\bm\mu\|_2^{4}\mu_{i_a}^ 2\mu_{j_a}^2\right).
\end{align}
Following the approach for $\widetilde{B}_5$,  
by \eqref{sqrtEva4} we bound $\widetilde{B}_4$ as
\begin{align}\label{tildeB4}
   \widetilde{B}_4 &\le C\sum_{i=1}^n\sum_{j\neq i}\sum_{k\neq i,j}p_{ij}p_{ik}
   \left( \|\bm\mu\|_2 \mu_{i}^{\frac{1}{2}}  \mu_{j}^{\frac{1}{2}} + \|\bm\mu\|_2^{4}\mu_{i}^ 2\mu_{j}^2\right)p_{jk}^{\frac12}\notag\\
   &\le C\sum_{i,j,k\in[n]}\left( \|\bm\mu\|_2 \mu_{i}^{\frac{5}{2}}\mu_{j}^{2} \mu_{k}^{\frac{3}{2}}+ \|\bm\mu\|_2^{4}\mu_{i}^{4}\mu_{j}^{\frac{7}{2}}
   \mu_{k}^{\frac{3}{2}}\right)\notag\\
   &=C\|\bm\mu\|_{2}^{3}(A_1^2+A_2^2),
\end{align}
which, together with \eqref{tildeB35}, implies that
\begin{align}\label{tildeB345}
\widetilde B_3+\widetilde B_4+\widetilde B_5\le 
 C \|\bm\mu\|_{2}^3(A_1^2+A_2^2+A_3^2).
\end{align}

We now turn to $\widetilde{B}_2$. By distinguishing cases where
 edges $e_a$ and $e_b$ are identical or distinct,
we decompose this term as
\begin{align}\label{tildeB20}
\widetilde B_2 &= \sum_{a,c\in [m]} h^2_ah_c  \E[\Delta_{ca}] + 
\sum_{\{a,b,c\}\subset[m]} h_ah_bh_c\E[\Delta_{ca}] \E[\Delta_{cb}]\notag\\
&=:\widetilde B_{21}+\widetilde B_{22},
\end{align}
where the simple fact $\E[\Delta_{ca}^2]=\E[\Delta_{ca}]$ is used.

For $\widetilde{B}_{21}$, analogous to the analysis of $\widetilde{B}_5$ and $\widetilde{B}_4$, we obtain that
\begin{align}\label{tildeB21}
\widetilde{B}_{21}&=\sum_{i=1}^n\sum_{j\neq i}\sum_{k\neq i,j}p_{ij}^2(1-p_{ij})^2p_{ik}(1-p_{ik})p_{jk}\notag\\
&\le \sum_{i,j,k\in[n]}p_{ij}^2p_{ik}p_{jk}\notag\\
&\le \sum_{i,j,k\in[n]}\mu_{i}^3\mu_{j}^3\mu_k^2 =\|\bm\mu\|_2^2\|\bm\mu\|_3^6.
\end{align}

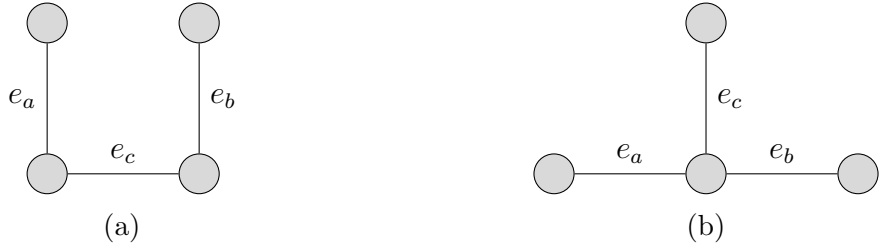
\begin{figure} 
  \centering
  \begin{subfigure}[b]{0.5\textwidth}
    \centering
    \begin{tikzpicture}
  % Nodes style
   \tikzset{vertex/.style={circle, draw, fill=gray!30, inner sep=0pt, minimum size=15pt}} 
  % Nodes
  \node[vertex] (l) at (1,2) {};
  \node[vertex] (i) at (1,0) {};
  \node[vertex] (j) at (3,0) {};
  \node[vertex] (k) at (3,2) {};
  % Edges
  \draw (l) -- node[left] {$e_a$} (i);
  \draw (i) -- node[above] {$e_c$} (j);
  \draw (j) -- node[right] {$e_b$} (k);
\end{tikzpicture}
 \caption{}
  \end{subfigure}\hspace{-1em}
  \begin{subfigure}[b]{0.5\textwidth}
    \centering
    \begin{tikzpicture}
  % Nodes style
   \tikzset{vertex/.style={circle, draw, fill=gray!30, inner sep=0pt, minimum size=15pt}} 
  % Nodes
  \node[vertex] (i) at (1,2) {};%$i_c$
  \node[vertex] (j) at (1,0) {};%$j_c$};
  \node[vertex] (k) at (3,0) {};%$k$
  \node[vertex] (l) at (-1,0) {};%$l$
  % Edges
  \draw (i) -- node[auto] {$e_c$} (j);
  \draw (j) -- node[above] {$e_b$} (k);
  \draw (j) -- node[above] {$e_a$} (l);
\end{tikzpicture}
    \caption{}
  \end{subfigure}
\caption{Edge $e_c$ shares a common vertex with two other edges $e_a$ and $e_b$.} 
\label{nodes}
\end{figure}

To bound $\widetilde{B}_{22}$, 
we only need to consider triples of edges $(e_a,e_b,e_c)$
where $e_c$ shares a common vertex with both $e_a$ and $e_b$.
Two configurations exist: either $e_c$ connects to $e_a$ and $e_b$ via distinct common vertices
(Figure \ref{nodes}(a)), or all three edges share a unique common vertex (Figure \ref{nodes}(b)). 
This yields that
\begin{align}\label{tildeB22}
\widetilde{B}_{22}&=\sum_{1\le i\ne j\le n}\sum_{k\neq i,j}\sum_{\ell\neq i,j,k} 
\big(p_{ij}(1-p_{ij})p_{ik}(1-p_{ik})p_{j\ell}(1-p_{j\ell})p_{i\ell}p_{jk}\notag\\
&\quad+p_{ij}(1-p_{ij})p_{ik}(1-p_{ik})p_{i\ell}(1-p_{i\ell})p_{jk}p_{j\ell}\big)\notag\\
&\le 2\sum_{i,j,k,\ell\in[n]}p_{ij}p_{ik}p_{j\ell}p_{i\ell}p_{jk}\notag\\
&\le 2 \sum_{i,j,k,\ell\in[n]} \mu_{i}^3\mu_{j}^3\mu_k^2\mu_{\ell}^2
=2\|\bm\mu\|_2^4\|\bm\mu\|_3^6.
\end{align}

Substituting \eqref{tildeB21} and \eqref{tildeB22} into \eqref{tildeB20} and applying the sparsity condition \eqref{eq:condion1}, we obtain
\begin{align} \label{B2}
    \widetilde B_2\le C  \big(\|\bm\mu\|_2^2\|\bm\mu\|_3^6+\|\bm\mu\|_2^4\|\bm\mu\|_3^6\big) \le C \|\bm\mu\|_2^4\|\bm\mu\|_3^6.
\end{align}
Invoking Lemma \ref{normsCS} with \(s = 2\) and \(t = 4\) gives
\[
\|\bm\mu\|_3^6 \le \|\bm\mu\|_{2} ^2 \|\bm\mu\|_4^4.
\]
Utilizing the following asymptotic relations
\[
\|\bm\mu\|_2=o(\|\bm\mu\|_2^2), \quad 
\|\bm\mu\|_4^4=o\Big(\|\bm\mu\|_{\frac{7}{2}}^{\frac{7}{2}}\Big),\quad
\|\bm\mu\|_2^4=o\Big(\|\bm\mu\|_{\frac{7}{4}}^{\frac{7}{2}}\Big),
\]
we have
\begin{align*}
\|\bm\mu\|_2\|\bm\mu\|_3^6\le \|\bm\mu\|_{2} ^3 \|\bm\mu\|_4^4
=o\Big(\|\bm\mu\|_{2} ^4\|\bm\mu\|_{\frac{7}{2}}^{\frac{7}{2}}\Big)
=o\big(A_5^2\big).
\end{align*}
Thus, from \eqref{B2} we can conclude that
\begin{align}\label{tildeB2}
\widetilde B_2= o\big(\|\bm\mu\|_{2}^3A_5^2\big).
\end{align}

Consider $\widetilde{B}_1 $.
Following a similar decomposition to \eqref{tildeB20}, we can rewrite 
\begin{align} \label{DecB1}
\widetilde{B}_1 & = \sum_{a,c\in [m]} h^2_a h_c\sqrt{\E[V_a^4]  \E[\Delta_{ca}]} 
                   +\sum_{\{a,b,c\}\subset[m]}h_ah_bh_c\sqrt{\E[V_a^2 V_b^2]\E[\Delta_{ca}]\E[\Delta_{cb}] }\notag\\
                &=: \widetilde{B}_{11}+\widetilde{B}_{12}.
\end{align}
Comparing $\widetilde{B}_{11}$ with the expression of $\widetilde{B}_4$, 
by \eqref{tildeB4} we obtain that 
\begin{align}\label{tildeB11}
\widetilde{B}_{11}
=o(\widetilde{B}_4)=o\big(\|\bm\mu\|_{2}^{3}\big(A_1^2+A_2^2\big)\big).
\end{align}
For $\widetilde{B}_{12}$, using Cauchy-Schwarz inequality and \eqref{sqrtEva4} gives
\begin{align*} 
\widetilde{B}_{12}&\le \sum_{\{a,b,c\}\subset[m]}h_ah_bh_c(\E[V_a^4])^{\frac{1}{4}} (\E [V_b^4])^{\frac{1}{4}}\sqrt{\E[\Delta_{ca}]
\E[\Delta_{cb} ] }\\
&\le C \|\bm\mu\|_2\sum_{\{a,b,c\}\subset[m]}h_ah_bh_c\left(\mu_{i_a}^{\frac{1}{4}}  \mu_{j_a}^{\frac{1}{4}} + \|\bm\mu\|_2^{\frac32}\mu_{i_a}\mu_{j_a}\right)\left(\mu_{i_b}^{\frac{1}{4}}  \mu_{j_b}^{\frac{1}{4}} + \|\bm\mu\|_2^{\frac32}\mu_{i_b}\mu_{j_b}\right)\\
&\quad \times \sqrt{\E[\Delta_{ca}]\E[\Delta_{cb} ] }.
\end{align*}
Then, analogously to \eqref{tildeB22}, we can proceed with
\begin{align}\label{tildeB120}
\widetilde{B}_{12}&\le C\|\bm\mu\|_2\sum_{1\le i\neq j\le n}\sum_{k\neq i,j}\sum_{\ell\neq i,j,k}
\Big[p_{ij}p_{ik}p_{j\ell}\Big(\mu_{i}^{\frac{1}{4}}  \mu_{k}^{\frac{1}{4}} + \|\bm\mu\|_2^{\frac32}\mu_{i}\mu_{k}\Big)
\Big(\mu_{j}^{\frac{1}{4}}  \mu_{\ell}^{\frac{1}{4}} + \|\bm\mu\|_2^{\frac32}\mu_{j}\mu_{\ell}\Big)p_{i\ell}^{\frac12}p_{jk}^{\frac12}\notag\\
&\quad +p_{ij}p_{ik}p_{i\ell}\Big(\mu_{i}^{\frac{1}{4}}  \mu_{k}^{\frac{1}{4}} + \|\bm\mu\|_2^{\frac32}\mu_{i}\mu_{k}\Big)
\Big(\mu_{i}^{\frac{1}{4}}  \mu_{\ell}^{\frac{1}{4}} + \|\bm\mu\|_2^{\frac32}\mu_{i}\mu_{\ell}\Big)
p_{jk}^{\frac12}p_{j\ell}^{\frac12}\Big]\notag\\
&\le C\|\bm\mu\|_2\sum_{i,j,k,\ell\in[n]}\Big[\mu_i^{\frac52}\mu_j^{\frac52}\mu_k^{\frac32}\mu_{\ell}^{\frac32} 
\Big(\mu_i^{\frac14}\mu_j^{\frac14}\mu_k^{\frac14}\mu_{\ell}^{\frac14}+
\|\bm\mu\|_2^{\frac32}\mu_i\mu_j^{\frac14}\mu_k\mu_{\ell}^{\frac14}+
\|\bm\mu\|_2^{3}\mu_i\mu_j\mu_k\mu_{\ell}\Big)\notag\\ 
&\quad +\mu_i^3\mu_j^2\mu_k^{\frac32}\mu_{\ell}^{\frac32}\Big(\mu_{i}^{\frac{1}{2}}\mu_{k}^{\frac{1}{4}}\mu_{\ell}^{\frac{1}{4}}+
\|\bm\mu\|_2^{\frac32}\mu_{i}^{\frac54}\mu_{k}\mu_{\ell}^{\frac14}+\|\bm\mu\|_2^{3}\mu_{i}^2\mu_k\mu_{\ell}\Big)\Big]\notag\\
&= C \|\bm\mu\|_{2}\Big( \|\bm\mu\|_{\frac{7}{4}}^{\frac{7}{2}} \|\bm\mu\|^{\frac{11}{2}}_{\frac{11}{4}}  + \|\bm\mu\|^{\frac{7}{4}}_{\frac{7}{4}}
         \|\bm\mu\|_{2}^{\frac{3}{2}}
        \|\bm\mu\|_{\frac{5}{2}}^{\frac{5}{2}}
        \|\bm\mu\|^{\frac{11}{4}}_{\frac{11}{4}}
        \|\bm\mu\|_{\frac{7}{2}}^{\frac{7}{2}}+\|\bm\mu\|_{2}^{3}\|\bm\mu\|^{5}_ {\frac{5}{2}}\|\bm\mu\|^{7}_{\frac{7}{2}}  \notag\\  
        & \quad  +\|\bm\mu\|_{\frac{7}{4}}^{\frac{7}{2}} \|\bm\mu\|^{2}_2 \|\bm\mu\|^{\frac{7}{2}}_{\frac{7}{2}}    + \|\bm\mu\|_{\frac{7}{4}}^{\frac{7}{4}}    \|\bm\mu\|_{2}^{\frac{7}{2}}  \|\bm\mu\|_{\frac{5}{2}}^{\frac{5}{2}} \|\bm\mu\|_{\frac{17}{4}}^{\frac{17}{4}}    +   \|\bm\mu\|_{2}^{5} \|\bm\mu\|_{\frac{5}{2}}^{5} \|\bm\mu\|_{5}^{5}\Big),
\end{align}
where the second inequality follows from $p_{ij}\le \mu_i\mu_j$ for all $i,j\in[n]$, together with symmetry. 
We denote the six terms in the parentheses above in turn by 
$\widetilde{B}_{12}^{(1)},\widetilde{B}_{12}^{(2)},\dots,\widetilde{B}_{12}^{(6)}$.
Clearly, from the relation $(\widetilde{B}_{12}^{(2)})^2=\widetilde{B}_{12}^{(1)}\widetilde{B}_{12}^{(3)}$, we have
\begin{equation}\label{B121}
  \widetilde{B}_{12}^{(2)}\le \frac12\big(\widetilde{B}_{12}^{(1)}+\widetilde{B}_{12}^{(3)}\big).
  \end{equation}
Setting \(s = \frac{7}{2}\) and \(t = 5\) in Lemma \ref{normsCS} yields that 
\[
\|\bm\mu\|^{\frac{17}{2}}_{\frac{17}{4}}  \le \|\bm\mu\|^{5}_5 \|\bm\mu\|^{\frac{7}{2}}_{\frac{7}{2}},
\]
which implies $(\widetilde{B}_{12}^{(5)})^2\le\widetilde{B}_{12}^{(4)}\widetilde{B}_{12}^{(6)}$.
Hence, we also have
\begin{equation}
  \widetilde{B}_{12}^{(5)}\le \frac12\big(\widetilde{B}_{12}^{(4)}+\widetilde{B}_{12}^{(6)}\big)
  =\frac12\|\bm\mu\|^{2}_2(A_5^2+A_4^2).  
\end{equation}
Furthermore, setting $(s,t)=(2,7/2)$ and $(2,5)$ in Lemma \ref{normsCS} yields  
\[
\|\bm\mu\|^{\frac{11}{2}}_{\frac{11}{4}}  \le \|\bm\mu\|^{2}_2 \|\bm\mu\|^{\frac{7}{2}}_{\frac{7}{2}} \quad \mbox{and} \quad 
\|\bm\mu\|^7_{\frac{7}{2}}\le \|\bm\mu\|^{2}_2\|\bm\mu\|^{5}_5,
\]
from which it follows that
\begin{equation}\label{B123}
\widetilde{B}_{12}^{(1)}\le \widetilde{B}_{12}^{(4)}, \quad \widetilde{B}_{12}^{(3)}\le \widetilde{B}_{12}^{(6)}.
\end{equation}
Combining \eqref{B121}-\eqref{B123} with \eqref{tildeB120}, we thus have
\begin{align} \label{tildeB12}
\widetilde B_{12} \le C\|\bm\mu\|_{2}\big(\widetilde{B}_{12}^{(4)}+\widetilde{B}_{12}^{(6)}\big)
     =C\|\bm\mu\|_{2}^3(A_4^2+A_5^2).
\end{align}
Hence,  by \eqref{tildeB2}-\eqref{tildeB11} and \eqref{tildeB12} we arrive at
\begin{align*}
\widetilde B_{1}+\widetilde B_{2}\le C\|\bm\mu\|_{2}^3(A_1^2+A_2^2+A_4^2+A_5^2).
\end{align*}
This, together with \eqref{tildeB345}, proves \eqref{B_K},
and thus completes the proof of Theorem \ref{Theorem:main}.
\end{proof}

To prove Theorem~\ref{Thm:AN}, we introduce the following auxiliary lemma, which
provides a reverse Cauchy-Schwarz inequality for positive vectors.

\begin{lemma}\label{lemma:reverse}
Let $\bm{x}=(x_1,\dots,x_n)$ and $\bm{y}=(y_1,\dots,y_n)$ be positive vectors (i.e., all entries are positive). 
For any real $s,t>0$,  we have
\begin{equation*}
\|\bm{x}\|_s^{s}\,\|\bm{y}\|_t^{t}
\le \frac{x_{\max}^{\,s/2}\, y_{\max}^{\,t/2}}{x_{\min}^{\,s/2}\, y_{\min}^{\,t/2}}
\Big(\sum_{i=1}^n x_i^{s/2} y_i^{t/2}\Big)^2.
\end{equation*}
\end{lemma}

\begin{proof}
Since all entries are positive, we have
\begin{align*}
x_{\min}^{s/2}y_{\min}^{t/2}\|\bm{x}\|_s^{s}\|\bm{y}\|_t^{t}
&=x_{\min}^{s/2}y_{\min}^{t/2}\sum_{i=1}^n\sum_{j=1}^n x_i^{s}y_j^{t} \\
&=\sum_{1\le i,j\le n} x_{\min}^{s/2}x_i^{s/2} x_i^{s/2}\cdot y_{\min}^{t/2}y_j^{t/2} y_j^{t/2}\\
&\le \sum_{1\le i,j\le n}x_j^{s/2}x_i^{s/2}x_{\max}^{s/2}\cdot y_i^{t/2}y_j^{t/2}y_{\max}^{t/2}\\
&=x_{\max}^{s/2}y_{\max}^{t/2}\Big(\sum_{i=1}^n x_i^{s/2}y_i^{t/2}\Big)^2,
\end{align*}
which proves the inequality as claimed. 
\end{proof}

Applying Lemma \ref{lemma:reverse} to $\bm{x}=\bm{y}=\bm{\mu}$ yields that for any $s,t>0$,
\begin{align}\label{eq:reverse2}
\|\bm\mu\|_s^s \,\|\bm\mu\|_t^t
\le \Big(\frac{\mu_{\max}}{\mu_{\min}}\Big)^{\frac{s+t}{2}} \|\bm\mu\|_{\frac{s+t}{2}}^{\,s+t}.
\end{align}

\begin{proof}[Proof of Theorem \ref{Thm:AN}]
Since the convergence in Kolmogorov distance implies the convergence in distribution,
Theorem \ref{Theorem:main} reduces our task to proving that under the conditions of Theorem \ref{Thm:AN},
\begin{align}\label{Aello}
A_{\ell}=o\Big(\|\bm\mu\|_2^{\frac{5} {2}}\|\bm\mu\|_3^6+\|\bm\mu\|_2^{\frac{9} {2}}\Big), \quad \ell=1,2,\dots,5.
\end{align}

The above relation holds for $A_3$ and $A_4$ only under the sparsity condition \eqref{eq:condion1}.
Indeed, by \eqref{normorder} we have
\[
A_3=\|\bm\mu\|_2^{\frac{5} {2}}\|\bm\mu\|_5^5=o\Big(\|\bm\mu\|_2^{\frac{9} {2}}\Big),
\qquad
A_4=\|\bm\mu\|_{2}^{\frac{3}{2}}\|\bm\mu\|_{\frac{5}{2}}^{\frac{5}{2}}\|\bm\mu\|_{5}^{\frac{5}{2}}=o\Big(\|\bm\mu\|_2^{\frac{9} {2}}\Big),
\]
and thus \eqref{Aello} holds for $\ell=3$ and 4.

In the following, we verify  \eqref{Aello}  for  $A_1,A_2$ and $A_5$ under either \eqref{heterocond} or \eqref{eq:condition2}.
We first assume \eqref{eq:condion1} and \eqref{heterocond}.

Consider $A_1$. Noting that $\|\bm\mu\|_{3/2}^{3/4} = O\big(\|\bm\mu\|_2^3\big)$ under \eqref{heterocond}
and $\|\bm\mu\|_{5/2}^{5/4} = o\big(\|\bm\mu\|_2 \big)$ under \eqref{eq:condion1}, we obtain 
\begin{align*}
A_1=\|\bm\mu\|_{\frac{3}{2}}^{\frac{3}{4}} \|\bm\mu\|_{\frac{5}{2}}^{\frac{5}{4}}
=o\big(\|\bm\mu\|_{2} ^ 4\big)
=o\Big(\|\bm\mu\|_{2} ^ {\frac{9}{2}}\Big),
\end{align*}
and thus \eqref{Aello} holds for $\ell=1$.

Similarly, using the facts $\|\bm\mu\|_{7/2}^{7/4}=o(\|\bm\mu\|_{3}^{3/2})$ and $\|\bm\mu\|_{4}^{2}=o(\|\bm\mu\|_{3}^{3/2})$, 
for $A_2$ we have 
\begin{align*}
A_2=\|\bm\mu\|_{\frac{3}{2}}^{\frac{3}{4}}\|\bm\mu\|_{2}^{\frac12}\|\bm\mu\|_{\frac{7}{2}}^{\frac{7}{4}}\|\bm\mu\|_{4}^{2}
=o\Big(\|\bm\mu\|_{2}^{\frac{7}{2}}\|\bm\mu\|_{3}^{3}\Big)
=o\Big( \Big(\|\bm\mu\|_2^{\frac{5} {2}}\|\bm\mu\|_3^6\Big)^{\frac12}
       \Big(\|\bm\mu\|_2^{\frac{9} {2}}\Big)^{\frac12} \Big).
\end{align*}
Then it follows from the basic inequality $\sqrt{xy}\le (x+y)/2$ for $x,y>0$ that \eqref{Aello} holds for $\ell=2$.

For $A_5$,  Lemma \ref{normsCS} with $s=3/2$ and $t=2$ yields 
\begin{equation*}
\|\bm\mu\|_{\frac{7}{4}}^{\frac{7}{2}} \le \|\bm\mu\|_{\frac{3}{2}} ^ {\frac{3}{2}} \|\bm\mu\|_2^2,   
\end{equation*}
which implies that 
\begin{equation}\label{A5bound}
A_5=\|\bm\mu\|_{\frac{7}{4}}^{\frac{7}{4}} \|\bm\mu\|_{\frac{7}{2}}^{\frac{7}{4}} 
\le\|\bm\mu\|_{\frac{3}{2}} ^ {\frac{3}{4}} \|\bm\mu\|_2\|\bm\mu\|_{\frac{7}{2}}^{\frac{7}{4}}.
\end{equation}
Thus, by \eqref{heterocond} and the fact $\|\bm\mu\|_{7/2}^{7/4}=o(\|\bm\mu\|_{3}^{3/2})$, we have
\begin{equation}\label{A5bound1}
A_5=o\Big(\|\bm\mu\|_{2}^{4}\|\bm\mu\|_3 ^ {\frac{3}{2}}\Big)
   =o\Big( \Big(\|\bm\mu\|_2^{\frac{5} {2}}\|\bm\mu\|_3^6\Big)^{\frac14}\Big(\|\bm\mu\|_2^{\frac{9} {2}}\Big)^{\frac34} \Big).
\end{equation}
Then it follows from the inequality $(xy^3)^{1/4}\le (x+3y)/4$ for $x,y>0$ that \eqref{Aello} holds for $\ell=5$.

Finally, under \eqref{eq:condion1} and \eqref{eq:condition2}, we invoke \eqref{eq:reverse2} to obtain
\begin{align}\label{eq:reverse3}
 \|\bm\mu\|_s^{\frac s2} \|\bm\mu\|_t^{\frac t2} = O\Big(\|\bm\mu\|_{2}^{{\frac {3(s+t)}{8}}}
 \|\bm\mu\|_{\frac{s+t}{2}}^{\frac{s+t}{2}}\Big).
\end{align}
We also consider $A_1,A_2$ and $A_5$ in turn.

For \(s = \frac{3}{2}\) and \(t = \frac{5}{2}\), the relation \eqref{eq:reverse3} gives
\begin{align}\label{A1bound}
A_1=\|\bm\mu\|_{\frac{3}{2}}^{\frac{3}{4}} \|\bm\mu\|_{\frac{5}{2}}^{\frac{5}{4}}=O\big(\|\bm\mu\|_{2}^{\frac72}\big)
=o\Big(\|\bm\mu\|_{2} ^ {\frac{9}{2}}\Big),
\end{align}
and thus \eqref{Aello} holds for $\ell=1$.

For $A_2$, by the second equality in \eqref{A1bound} and the simple fact $\|\bm\mu\|_{4}^{2}=o(\|\bm\mu\|_{3}^{3/2})$, we have
\[
A_2=\|\bm\mu\|_{\frac{3}{2}}^{\frac{3}{4}}\|\bm\mu\|_{2}^{\frac12}\|\bm\mu\|_{\frac{7}{2}}^{\frac{7}{4}}\|\bm\mu\|_{4}^{2}
=o\Big(\|\bm\mu\|_{\frac{3}{2}}^{\frac{3}{4}}\|\bm\mu\|_{2}^{\frac12}\|\bm\mu\|_{\frac{5}{2}}^{\frac{5}{4}}\|\bm\mu\|_{4}^{2}\Big)
=o\Big(\|\bm\mu\|_{2}^{4}\|\bm\mu\|_{3}^{\frac32}\Big),
\]
which matches \eqref{A5bound1}, and thus \eqref{Aello} holds for $\ell=2$.

Consider $A_5$. Note that \eqref{A5bound} holds under \eqref{eq:condion1},
and Lemma \ref{normsCS} with $s=2$ and $t=3$ gives
$\|\bm\mu\|_{5/2}^{5/2}\le \|\bm\mu\|_2\|\bm\mu\|_3^{3/2}$.
Applying \eqref{eq:reverse3} with \(s = \frac{3}{2}\) and \(t = \frac{7}{2}\)
to \eqref{A5bound},  we thus have
\[
A_5=o\Big(\|\bm\mu\|_{2} ^ {\frac{23}{8}}\|\bm\mu\|_{\frac{5}{2}}^{\frac{5}{2}}\Big)=o\Big(\|\bm\mu\|_{2}^3\|\bm\mu\|_{\frac{5}{2}}^{\frac{5}{2}}\Big)
=o\Big(\|\bm\mu\|_{2}^{4}\|\bm\mu\|_3 ^ {\frac{3}{2}}\Big),
\]
which also matches \eqref{A5bound1}. 
This proves \eqref{Aello} for $\ell=5$, 
and completes the proof of Theorem \ref{Thm:AN}.
\end{proof}

\section*{Acknowledgment}
We wish to thank two anonymous referees for their constructive comments that help improve the quality of the paper.

\bibliographystyle{plain}
\bibliography{reference}

@article{alon2007,
 title={Network motifs: theory and experimental approaches},
 author={Alon, U.},
 journal={Nat. Rev. Genet.},
 volume={8},
 pages={450--461},
year={2007},
}

@article{barabasi1999,
  title={Emergence of scaling in random networks},
  author={Barab{\'a}si, A. -L. and Albert, R.},
  journal={Science},
  volume={286},
  pages={509--512},
  year={1999},
}

@book{barbour2005introduction,
  title={An Introduction to Stein's Method},
  author={Barbour, A. D. and Chen, L. H. Y.},
  year={2005},
  publisher={World Scientific},
  address={Singapore}
}

@book{bollobas2001,
    title={Random Graphs, 2nd Edition},
    author={Bollob\'{a}s, B.},
    year={2001},
    publisher={Cambridge University Press},
    address={Cambridge},
}

@article{bullmore2009,
 title={Complex brain networks: graph theoretical analysis of structural and functional systems},
 author={Bullmore, E. and Sporns, O.},
 journal={Nat. Rev. Neurosci.},
 volume={10},
 pages={186--198},
 year={2009},
}

@article{Chang2024,
title={Edge differentially private estimation in the $\beta$-model via jittering and method of moments},
author={Chang, J. and Hu, Q. and Kolaczyk, E. D. and Yao, Q. and Yi, F.},
journal={Ann. Statist.},
year={2024},
number = {2},
volume={52},
pages={508--528},
}

@article{Chatterjee2011RandomGW,
  title={Random graphs with a given degree sequence},
  author={Chatterjee, S. and Diaconis, P. and Sly, A.},
  journal={Ann. Appl. Probab.},
  year={2011},
  number = {4},
  volume={21},
  pages={1400--1435},
  url={https://api.semanticscholar.org/CorpusID:17203741}
}

@article{Chen2021,
author = {Chen, M. and Kato, K. and Leng, C.},
title = {Analysis of networks via the sparse $\beta$-model},
journal = {J. R. Stat. Soc. B‌‌},
volume = {83},
number = {5},
pages = {887-910},
doi = {https://doi.org/10.1111/rssb.12444},
year = {2021}
}

@article{clauset2009,
title={Power-law distributions in empirical data},
author={Clauset, A. and Shalizi, C. R. and Newman, M. E. J.},
journal={SIAM Review},
year={2009},
volume={51},
number={4},
pages={661--703},
}

@article{eichelsbacher2023Kol,
title={Kolmogorov bounds for decomposable random variables and subgraph counting by the {S}tein-{T}ikhomirov method},
author={Eichelsbacher, P. and Redno{\ss}, B.},
year={2023},
journal={Bernoulli},
volume={29},
number={3},
pages={1821--1848},
}

@article{eichelsbacher2023simplified,
  title={A simplified second-order {G}aussian {P}oincar{\'e} inequality in discrete setting with applications},
  author={Eichelsbacher, P. and Redno{\ss}, B. and Th{\"a}le, C. and Zheng, G.},
  journal={Ann. Inst. H. Poincar\'e Probab. Statist.},
  volume={59},
  number={1},
  pages={271--302},
  year={2023}
}

@article{Estrada2010,
  title = {Quantifying network heterogeneity},
  author = {Estrada, E.},
  journal = {Phys. Rev. E},
  volume = {82},
  issue = {6},
  pages = {066102},
  year = {2010},
  publisher = {American Physical Society},
  doi = {10.1103/PhysRevE.82.066102},
}

@article{holland1981,
author={Holland, P. W. and Leinhardt, S.},
year={1981},
title={An exponential family of probability distributions for directed graphs},
journal={J. Amer. Stat. Assoc.},
volume={76}, 
pages={33--50},
}

@book{janson2000random,
	title={Random Graphs},
	author={Svante, J. and \L{}uczak, T. and Ruci\'nski, A.},
	year={2000},
	publisher={John Wiley \& Sons},
	address={New York}
}

@article{jin2015,
author={Jin, J.},
year={2015},
title={Fast community detection by {SCORE}},
journal={Ann. Statist.},
volume={43},
number={1},
pages={57--89},
}

@article{Karwa2016,
title={Inference using noisy degrees-Differentially private $\beta$-model and synthetic graphs},
author={V. Karwa and A. Slavkovi\'c},
journal={Ann. Statist.},
volume={44},
year={2016},
pages={87--112}
}

@article{Krokowski2017Rademacher,
author = {Krokowski, K. and Reichenbachs, A. and Th\"{a}le, C.},
year = {2016},
pages = {763--803},
title = {Berry-{E}sseen bounds and multivariate limit theorems for functionals of {R}ademacher sequences},
volume = {52},
number = {2},
journal = {Ann. Inst. H. Poincar\'e Probab. Statist.},
doi = {10.1214/14-AIHP652}
}

@article{Kai2017DiscreteMalliavin,
author = {Krokowski, K. and Reichenbachs, A. and Th{\"a}le, C.},
title = {{D}iscrete {M}alliavin–{S}tein method: {B}erry-{E}sseen bounds for random graphs and percolation},
volume = {45},
journal = {Ann. Probab.},
number = {2},
publisher = {Institute of Mathematical Statistics},
pages = {1071--1109},
year = {2017},
doi = {10.1214/15-AOP1081},
}

@article{milo2002,
author = {Milo, R. and Shen-Orr, S. and Itzkovitz, S. and Kashtan, N. and Chklovskii, D. and Alon, U.},
title = {Network Motifs: Simple Building Blocks of Complex Networks},
journal = {Science},
volume ={298},
year={2002},
pages = {824--827},
}

@book{nualart2006,
 title = {The Malliavin Calculus and Related Topics},
 author = {Nualart, D.},
 publisher={Springer}, 
 address={Berlin}, 
 year={2006}, 
}

@article{newman2009random,
 title = {Random Graphs with Clustering},
  author = {Newman, M. E. J.},
  journal = {Phys. Rev. Lett.},
  volume = {103},
  issue = {5},
  pages = {058701},
  numpages = {4},
  year = {2009},
  publisher = {American Physical Society},
  doi = {10.1103/PhysRevLett.103.058701},
  url = {https://link.aps.org/doi/10.1103/PhysRevLett.103.058701}
}

@book{newman2018,
 title = {Networks, 2nd Edition},
 author = {Newman, M. E. J.},
 publisher={Oxford University Press}, 
 address={Oxford}, 
 year={2018}, 
}

@article{newman2001,
title = {Random graphs with arbitrary degree distributions and their applications},
author = {Newman, M. E. J. and Strogatz, S. H. and Watts, D. J.},
journal = {Phys. Rev. E},
volume = {64},
year = {2001},
pages = {026118},
}

@book{Nourdin_Peccati_2012, 
address={Cambridge}, 
title={Normal Approximations with Malliavin Calculus: From Stein’s Method to Universality}, 
publisher={Cambridge University Press}, 
author={Nourdin, I. and Peccati, G.}, 
year={2012}, 
}

@article{Privault2008StochasticAO,
 author = {N. Privault},
title = {{Stochastic analysis of Bernoulli processes}},
volume = {5},
journal = {Probab. Surv.},
publisher = {Institute of Mathematical Statistics and Bernoulli Society},
pages = {435 -- 483},
year = {2008},
doi = {10.1214/08-PS139},
}

@article{Privault2020,
author = {N. Privault and Serafin, G.},
title = {Normal approximation for sums of weighted {$U$}-statistics ---application to 
{K}olmogorov bounds in random subgraph counting},
journal = {Bernoulli},
year = {2020},
volume={26},
number={1},
pages={587--615}
}

@article{Ribeiro2021,
author = {Ribeiro, P. and Paredes, P. and Silva, M. E. P. and Aparicio, D. and Silva, F.},
title = {A Survey on Subgraph Counting: Concepts, Algorithms, and Applications to Network Motifs and Graphlets},
year = {2021},
issue_date = {March 2022},
publisher = {Association for Computing Machinery},
address = {New York, NY, USA},
volume = {54},
number = {2},
issn = {0360-0300},
url = {https://doi.org/10.1145/3433652},
doi = {10.1145/3433652},
journal = {ACM Comput. Surv.},
pages= {Article 28},
}

@article{Rinaldo2013,
author = "Rinaldo, A. and Petrovi\'c, S. and Fienberg, S. E.",
doi = "10.1214/12-AOS1078",
fjournal = "The Annals of Statistics",
journal = "Ann. Statist.",
number = "3",
pages = "1085--1110",
publisher = "The Institute of Mathematical Statistics",
title = "Maximum lilkelihood estimation in the $\beta$-model",
url = "https://doi.org/10.1214/12-AOS1078",
volume = "41",
year = "2013"
}

@article{robins2007,
author={Robins, G. and Pattison, P. and Kalish, Y. and Lusher, D.},
title={An introduction to exponential random graph models for social networks},
year={2007},
journal={Soc. Networks},
volume={29},
pages={173--191},
}

@article{rollin2022,
author={R\"{o}llin, A.}, 
title={Kolmogorov bounds for the normal approximation of the number of triangles in the {Erd\H{o}s-R\'enyi} random graph},
year={2022},
journal={Prob. Eng. Inf. Sci.},
volume={36},
pages={747--773},
}

@article{Rucinski1988WhenAS,
  title={When are small subgraphs of a random graph normally distributed?},
  author={Ruci\'nski, A.},
  journal={Probab. Theory Relat. Fields},
  year={1988},
  volume={78},
  pages={1--10},
  url={https://api.semanticscholar.org/CorpusID:119856637}
}

@inproceedings{Teixeira2015,
author = {Teixeira, C. H. C. and Fonseca, A. J. and Serafini, M. and Siganos, G. and Zaki, M. J. and Aboulnaga, A.},
title = {Arabesque: a system for distributed graph mining},
booktitle = {Proceedings of the 25th Symposium on Operating Systems Principles},
year = {2015},
address = {New York},
doi = {10.1145/2815400.2815410},
pages = {425–440},
numpages = {16},
location = {Monterey, California},
series = {SOSP '15}
}

@article{watts1998,
author={Watts, D. J. and Strogatz, S. H.}, 
title={Collective dynamics of `small-world' networks},
year={1998},
journal={Nature},
volume={393},
pages={440--442},
}

@article{Yan2013clt,
author = {Yan, T. and Xu, J.},
title = "{A central limit theorem in the $\beta$-model for undirected random graphs with a diverging number of vertices}",
journal = {Biometrika},
volume = {100},
number = {2},
pages = {519-524},
year = {2013},
doi = {10.1093/biomet/ass084},
}

@article{butzek2024nonuni,
title={Non-uniform {B}erry-{E}sseen bounds for {G}aussian, {P}oisson and {R}ademacher processes}, 
author={Butzek, M. and Eichelsbacher, P.},
year={2024},
volume= {arXiv:2409.09439v1},
journal={arXiv preprint},
url={https://arxiv.org/abs/2409.09439}, 
}

@article{Gesine2010SteinRademacher,
author = {Nourdin, I. and Peccati, G. and Reinert, G.},
title = {{S}tein's method and Stochastic Analysis of {R}ademacher Functionals},
volume = {15},
journal = {Electron. J. Probab.},
pages = {1703--1742},
year = {2010},
doi = {10.1214/EJP.v15-823},
}

@article{bhattacharya2023fluctuations,
  title={Fluctuations of subgraph counts in graphon based random graphs},
  author={Bhattacharya, B. B. and Chatterjee, A. and Janson, S.},
  journal={‌Comb. Probab. Comput.‌},
  volume={32},
  number={3},
  pages={428--464},
  year={2023},
  doi={10.1017/S0963548322000335},
}

@article{vanderhofstad2020assortativity,
  title={Limit theorems for assortativity and clustering in null models for scale-free networks},
  author={van der Hofstad, R. and van der Hoorn, P. and Litvak, N. and Stegehuis, C.},
  journal={Adv. Appl. Probab.},
  volume={52},
  number={4},
  pages={1035--1084},
  year={2020},
  doi={10.1017/apr.2020.42},
}

@article{KarrerNewman2011,   
 title={Stochastic blockmodels and community structure in networks},     
 author={Karrer, B. and Newman, M. E. J.},
 journal={Phys. Rev. E}, 
 volume={83}, 
 number={1}, 
 pages={016107}, 
 year={2011}, 
 doi={10.1103/PhysRevE.83.016107} 
}

@article{Abbe2018,
  author = {Abbe, E.},
  title   = {Community Detection and Stochastic Block Models: Recent Developments},
  journal = {J. Mach. Learn. Res.},
  year    = {2018},
  volume  = {18},
  number  = {177},
  pages   = {1--86},
  url     = {http://jmlr.org/papers/v18/16-480.html}
}

@article{MosselNeemanSly2015,
  title   = {Reconstruction and estimation in the planted partition model},
  author = {Mossel, E. and Neeman, J. and Sly, A.},
  journal = {Probab. Theory Relat. Fields},
  volume  = {162},
  number  = {3--4},
  pages   = {431--461},
  year    = {2015},
  doi     = {10.1007/s00440-014-0576-6}
}

@article{Hladky2021,
author = {Hladk\'{y}, J. and Pelekis, C. and \v{S}ileikis, M.},
title = {A limit theorem for small cliques in inhomogeneous random graphs},
journal = {J. Graph Theory},
volume = {97},
number = {4},
pages = {578-599},
doi = {https://doi.org/10.1002/jgt.22673},
url = {https://onlinelibrary.wiley.com/doi/abs/10.1002/jgt.22673},
year = {2021}
}
\end{document}